\documentclass[a4paper]{amsart}

\title{Recurrence Relations for Graph Polynomials
        on Bi-iterative Families of Graphs}
\thanks{The first author would like to acknowlege support by the Austrian National Research
    Network S11403-N23 (RiSE) of the Austrian Science Fund (FWF) and
    by the Vienna Science and Technology Fund (WWTF) through
    grants PROSEED, ICT12-059, and VRG11-005.
}

\author{Tomer Kotek}
\address{Institute for information systems \\ Vienna University of Technology}
\email{kotek@tuwien.forsyte.at}

\author{Johann A. Makowsky}
\address{Department of Computer Science \\ Technion --- Israel Institute of Technology \\ Haifa, Israel}
\email{janos@cs.technion.ac.il}

\usepackage{calc}
\usepackage{ifthen}
\usepackage{epsfig}
\usepackage{graphicx}
\usepackage{latexsym}
\usepackage{amssymb,amsmath}

\newcommand{\MSOL}{MSOL}
\newcommand{\MS}{MS}
\newcommand{\FOL}{FOL}
\newcommand{\CMSOL}{CMSOL}

\newcommand{\CMSOLs}{CMSOL~}
\newcommand{\MSOLs}{MSOL~}
\newcommand{\Ct}{C${}^2$}
\newcommand{\Ctem}{C${\,}^2$}

\makeatother
\numberwithin{equation}{section}
\numberwithin{figure}{section}
\theoremstyle{plain}
\newtheorem{thm}{\protect\theoremname}
  \theoremstyle{definition}
  \newtheorem{defn}[thm]{\protect\definitionname}
  \theoremstyle{definition}
  \newtheorem{example}[thm]{\protect\examplename}
  \theoremstyle{plain}
  \newtheorem{lem}[thm]{\protect\lemmaname}
  \theoremstyle{remark}
  \newtheorem{rem}[thm]{\protect\remarkname}
  \theoremstyle{plain}
  \newtheorem{cor}[thm]{\protect\corollaryname}
  \theoremstyle{definition}
  \newtheorem{problem}[thm]{\protect\problemname}

  \providecommand{\corollaryname}{Corollary}
  \providecommand{\definitionname}{Definition}
  \providecommand{\examplename}{Example}
  \providecommand{\lemmaname}{Lemma}
  \providecommand{\problemname}{Problem}
  \providecommand{\remarkname}{Remark}
\providecommand{\theoremname}{Theorem}

\begin{document}

\begin{abstract}
We show that any graph polynomial from a wide class of graph polynomials
yields a recurrence relation on an infinite class of families of graphs. The
recurrence relations we obtain have coefficients which themselves
satisfy linear recurrence relations. We give explicit applications to the Tutte polynomial
and the independence polynomial. Furthermore, we get that for any
sequence $a_{n}$ satisfying a linear recurrence with constant coefficients,
the sub-sequence corresponding to square indices $a_{n^{2}}$ and
related sub-sequences satisfy recurrences with recurrent coefficients. 
\end{abstract}

\maketitle
\sloppy


\section{Introduction}

Recurrence relations are a major theme in the study of graph polynomials.
As early as 1972, N. L. Biggs, R. M. Damerell and D. A. Sands \cite{ar:BiggsDS72}
studied sequences of Tutte polynomials which are C-finite, i.e. satisfy
a homogenous linear recurrence relation with constant coefficients (or equivalently, sequences of 
coefficients of rational power series). 
More recently, M. Noy and A. Rib\'o \cite{ar:NoyRibo04} proved that
over an infinite class of {\em recursively constructible families of graphs},
which includes e.g. paths, cycles, ladders and wheels, the Tutte polynomial
is C-finite (see also \cite{ar:BoueyGOP2012}). The Tutte polynomials
of many recursively constructible families of graphs received special
treatment in the literature. Moreover, the Tutte polynomial can be
defined through its famous deletion-contraction recurrence relation.

Similar recurrence relations have been studied for other graph polynomials,
e.g. for the independence polynomial see e.g. \cite{pr:LevitMandrescu2005,ar:Wang2011}.
E. Fischer and J. A. Makowsky \cite{pr:FischerMakowsky08} extended
the result of Noy and Rib\'o to an infinite class of graph polynomials
definable in Monadic Second Order Logic (\MSOL), which includes the
matching polynomial, the independence polynomial, the interlace polynomial,
the domination polynomial and many of the graph polynomials which
occur in the literature. \cite{pr:FischerMakowsky08} applies to the wider class
of {\em iteratively constructible graph families}. 
The class of \MSOL-polynomials and variations
of it were studied with respect to their combinatorial and computational
properties e.g. in \cite{ar:CourcelleMR2000,pr:GodlinKotekMakowsky08,pr:KotekMakowskyCSL12,ar:Makowsky2005}.
L. Lov{\'a}sz treats \MSOL-definable graph invariants in \cite{bk:Lovasz2012}.

In this paper we consider recurrence relations of graph polynomials
which go beyond C-finiteness. A sequence is {\em\Ctem-finite} if
it satisfies a linear recurrence relation with C-finite coefficients.
We start by investigating the set of \Ct-finite sequences. The tools
we develop apply to sparse sub-sequences of C-finite sequences. While
C-finite sequences have received considerable attention in the literature,
cf. e.g. \cite[Chapter 4]{bk:StanleyI}, and it is well-known that
taking a linear sub-sequence $a_{qn+r}$ of a C-finite sequence $a_{n}$
yields again a C-finite sequence, it seems other types of sub-sequences
have not been systematically studied. We show the following:
\begin{thm}
\label{mainth:2} Let $a_{n}$ be a C-finite over $\mathbb{C}$. Let
$c\in\mathbb{N}^{+}$ and $d,e\in\mathbb{Z}$. Then the sequence 
\[
b_{n}=a_{c\binom{n}{2}+dn+e}
\]
 is \Ctem-finite.
\end{thm}
In particular, $a_{n^{2}}$ and $a_{\binom{n}{2}}$ are \Ct-finite.
The proof of Theorem \ref{mainth:2} is given in Section \ref{se:mainth:2-proof}.
As an explicit example, we consider the Fibonacci numbers in Section
\ref{se:fibonacci}. 

Next, we show \MSOL-polynomials satisfy \Ct-recurrences on appropriate
families of graphs. In Section \ref{se:bi-iterative} we introduce
the notion of {\em bi-iteratively constructible graph families},
or {\em bi-iterative families} for short. In Section \ref{se:msol}
we recall from the literature the definitions of two related classes
of \MSOL-polynomials and introduce a powerful theorem for them. The
main theorem of the paper is:
\begin{thm}
[Informal] \label{mainth:1} \MSOL-polynomials satisfy \Ctem-finite
recurrences on bi-iterative families.
\end{thm}
Theorem \ref{mainth:1} shows the existence of the desired recurrence
relations. The exact statement Theorem \ref{mainth:1}, namely Theorem \ref{th:main-formal},
is given in Section \ref{se:mainth:1-proof} together with the proof. In Section \ref{se:examples}
we compute explicit \Ct-recurrences for the Tutte polynomial and
the independence polynomial. Finally, in Section \ref{se:conclusion}
we conclude and discuss future research.

\section{\Ct-finite sequences}

In this section we define the recurrence relations we are interested
in and give useful properties of sequences satisfying them.
\begin{defn}
\label{def:Ct-finite} Let $\mathbb{F}$ be a field. Let $a_{n}:\, n\in\mathbb{N}$
be a sequence over $\mathbb{F}$. 
\begin{enumerate}
\item $a_{n}$ is C-finite if there exist $s\in\mathbb{N}$ and $c^{(0)},\ldots,c^{(s)}\in\mathbb{F}$,
$c^{(s)}\not=0$, such that for every $n\geq s$, 
\[
c^{(s)}a_{n+s}=c^{(s-1)}a_{n+s-1}+\cdots+c^{(0)}a_{n}\,.
\]
We may assume w.l.o.g. that $c^{(s)}=1$. 
\item $a_{n}$ is P-recursive if there exist $s\in\mathbb{N}$ and $c_{n}^{(0)},\ldots,c_{n}^{(s)}$
which are polynomials in $n$ over $\mathbb{F}$, such that for every
$n$ we have $c_{n}^{(s)}\not=0$, and for every $n\geq s$, 
\begin{gather}
c_{n}^{(s)}a_{n+s}=c_{n}^{(s-1)}a_{n+s-1}+\cdots+c_{n}^{(0)}a_{n}\,.\label{eq:rec-def}
\end{gather}

\item $a_{n}$ is \Ct-finite if there exist $s\in\mathbb{N}$ and C-finite
sequences $c_{n}^{(0)},\ldots,c_{n}^{(s)}$, such that for every $n$
we have $c_{n}^{(s)}\not=0$, and for every $n\geq s$ , Eq. (\ref{eq:rec-def})
holds.
\end{enumerate}
\end{defn}
P-recursive (holonomic) sequences have been studied in their own right, but also
as the coefficients of Differentially finite generating functions
\cite{bk:StanleyII}, see also \cite{bk:A=B}. 
\begin{example}
[\Ct-finite sequences] Sequences with \Ct-finite recurrences emerge
in various areas of mathematics. 
\begin{enumerate}
\item The $q$-derangement numbers $d_{n}(q)$ are polynomials in $q$ related
to the set of derangements of size $n$. A formula for computing them
in analogy to the standard derangement numbers was found by I. Gessel
\cite{ar:Gessel93} and M. L. Wachs \cite{ar:Wachs1989}. This formula
implies that the following \Ct-recurrence holds:
\[
d_{n}(q)=(q^{n}+[n])d_{n-1}(q)-q^{n}[n]d_{n-2}(q)\,,
\]
see also \cite{ar:Desarmenien1993}. We denote here $[n]=1+q+q^2+\cdots +q^{n-1}$. 
\item In knot theory, the colored Jones polynomial of a framed knot $\mathcal{K}$
in $3$-space is a function from such knots to polynomials. The colored
Jones function of the $0$-framed right-hand trefoil satisfies the
following \Ct-recurrence \cite{ar:Garoufalidis05}:
\begin{eqnarray*}
J_{\mathcal{K}}(n) &=& \frac{x^{2n-2}+x^{8}y^{4n}-y^{n}-x^{2}y^{2n}}{x(x^{2n}y-x^{4}y^{n})}J_{\mathcal{K}}(n-1)\\
& & +\frac{x^{8}y^{4n}-x^{6}y^{2n}}{x^{4}y^{n}-x^{2n}y}J_{\mathcal{K}}(n-2)
\end{eqnarray*}
with $x=q^{1/2}$ and $y=x^{-2}$. See \cite{ar:Gelca2002} and \cite{ar:GelcaSain2003}
for more examples. 
\end{enumerate}
\end{example}
\begin{lem}
[Properties]~
\begin{enumerate}
\item Every C-finite sequence is P-recursive.
\item Every P-recursive sequence is \Ct-finite.
\item For every C-finite sequence $a_{n}$, there exists $\alpha\in\mathbb{N}$
such that $a_{n}\leq\alpha^{n}$ for every large enough $n$. 
\item For every P-recursive sequence $a_{n}$, there exists $\alpha\in\mathbb{N}$
such that $a_{n}\leq n!^{\alpha}$ for every large enough $n$. 
\item For every \Ctem-finite sequence $a_{n}$, there exists $\alpha\in\mathbb{N}$
such that $a_{n}\leq\alpha^{n^{2}}$ for every large enough $n$. 
\end{enumerate}
\end{lem}
\begin{proof}
1 and 2 follow directly from Definition \ref{def:Ct-finite}. 3, 4
and 5 can be proven easily by induction on $n$. 
\end{proof}

The following will be useful, see e.g. \cite{bk:StanleyI}: 
\begin{lem}
[Closure properties] \label{lem:closure} The C-finite sequences are
closed under:
\begin{enumerate}
\item Finite addition;
\item Finite multiplication;
\item Given a C-finite sequence $a_{n}$, taking sub-sequences $a_{tn+s}$,
$t\in\mathbb{N}^{+}$ and $s\in\mathbb{Z}$. 
\end{enumerate}
\end{lem}

The sets of C-finite sequences and P-recursive sequences form rings
with respect to the usual addition and multiplication. However, they
are not integral domains. For every $i\leq p$ and every $n$ let
\[
\mathbb{I}_{n\equiv i\,(mod\, p)}=\begin{cases}
1 & n\equiv i\,(mod\, p)\\
0 & n\not\equiv i\,(mod\, p)
\end{cases}
\]
For every $i\leq p$, $\mathbb{I}_{n\equiv i\,(mod\, p)}$ is C-finite.
While each of $\mathbb{I}_{n\equiv0\,(mod\,2)}$ and $\mathbb{I}_{n\equiv1\,(mod\,2)}$
is not identically zero, their product is. 
This obsticale  complicates our proofs in the sequel, and is overcome using
a classical theorem on the zeros of C-finite sequences:
\begin{thm}
[Skolem-Mahler-Lech Theorem] If $a_{n}$ is C-finite, then there
exist a finite set $I\subseteq\mathbb{N}$, $n_{1},p\in\mathbb{N}$,
and $P\subseteq\{0,\ldots,p-1\}$ such that 
\[
\{n\mid a_{n}=0\}=I\cup\bigcup_{i\in P}\{n\mid n>n_{1},\, n\equiv i\,(mod\, p)\}\,.
\]
\end{thm}
\begin{rem}
Recently J. P. Bell, S. N. Burris and K. Yeats \cite{ar:BellBurrisYeats2012}
extended the Skolem-Mahler-Lech theorem extends to a Simple P-recursive
sequences, P-recursive sequences where the leading coefficient is
a constant. 
\end{rem}

\subsection{C-finite matrices}
A notion of sequences of matrices whose entries are C-finite sequences will be useful. 
We define this exactly and prove some properties of these matrices sequences. 
\begin{defn}
Let $r\in\mathbb{N}$ and let $\left\{ A_{n}\right\} _{n=1}^{\infty}$
be a sequence of $r\times r$ matrices over a field $\mathbb{F}$.
We say $\left\{ A_{n}\right\} _{n=1}^{\infty}$ a {\em C-finite matrix
sequence} if for every $1\leq i,j\leq r$, the sequence $A_{n}[i,j]$
is C-finite.
\end{defn}

\begin{lem}
\label{lem:c-finiteness-matrices} Let $r,n_{0}\in\mathbb{N}$ and
let $\left\{ A_{n}\right\} _{n=n_{0}}^{\infty}$ be a C-finite matrix
sequence of $r\times r$ matrices over $\mathbb{C}$. The following
hold:
\begin{enumerate}
\item The sequence $\left\{ A_{n}^{T}\right\} _{n=1}^{\infty}$ is an C-finite
matrix sequence.
\item The sequence $\left\{ \left|A_{n}\right|\right\} _{n=1}^{\infty}$
is in C-finite. 
\item For any fixed $i,j$, the sequence of consisting of the $(i,j)$-th
cofactor of $A_{n}$ is C-finite, and the sequence $\left\{ C_{n}\right\} _{n=1}^{\infty}$
of matrices of cofactors of $A_{n}$ is an C-finite matrix sequence. 
\item There exist $n_{1}$ and $p$ such that, for every $0\leq i\leq p-1$
and $n\in\mathbb{N}^{+}$, $\left|A_{i+n_{1}}\right|=0$ iff \textup{$\left|A_{pn+i+n_{1}}\right|=0$.}
\item Let $n_{1},p\in\mathbb{N}$. If \textup{$\left|A_{pn+i+n_{1}}\right|\not=0$
for every $n\in\mathbb{N}^{+}$, then the sequence of matrices of
the form $\left(A_{pn+i+n_{1}}\right)^{-1}$ is an} C-finite matrix
sequence. 
\end{enumerate}
\end{lem}
\begin{proof}
~
\begin{enumerate}
\item Immediate. 
\item The determinant is a polynomial function of the entries of the matrix,
so it is C-finite by the closure of the set of C-finite sequences
to finite addition and multiplication. 
\item The cofactor is a constant times a determinant, so again it is C-finite. 
\item This follows from the Lech-Mahler-Skolem property of C-finite sequences
and from the fact that the determinant is a C-finite sequence. 
\item The transpose of the matrix of cofactors $C_{n}$ of $A_{n}$ is an
C-finite matrix sequence by the above. Since $\left|A_{pn+i+n_{1}}\right|\not=0$
for every $n\in\mathbb{N}^{+}$, then the $\left(A_{pn+i+n_{1}}\right)^{-1}=\frac{1}{\left|A_{pn+i+n_{1}}\right|}C_{n}$
is well-defined and an C-finite matrix sequence. 
\end{enumerate}
\end{proof}

\begin{lem}
\label{lem:power-of-mat} Let $M$ be an $r\times r$ matrix. Let
$c,d\in\mathbb{Z}$ with $c>0$. Let 
\[
M_{n}=\begin{cases}
M^{cn+d}, & cn+d\geq0\\
0, & \mbox{otherwise}
\end{cases}
\]
The sequence $M_{n}$ is a C-finite matrix sequence. \end{lem}
\begin{proof}
Let

\[
\chi(\lambda)=\sum_{t=0}^{r}e_{t}\lambda^{t}
\]
be the characteristic polynomial of $M^{c}$, with $e_{r}\not=0$.
By the Cayley-Hamilton theorem, $\chi(M^{c})=0$, so 
\begin{gather}
0=\sum_{t=1}^{r}e_{t}M^{ct}\label{eq:rec-ch}
\end{gather}
with $e_{r}\not=0$. If $d\geq0$, then by multiplying Eq. (\ref{eq:rec-ch})
by $M^{d}$ and setting $t=n$, we get that for every $i,j$, the
entry $(i,j)$ in the sequence of matrices $M_{n}:\, n\in\mathbb{N}$
satisfies the recurrence 
\[
M_{n}[i,j]=-\sum_{t=1}^{r-1}\frac{e_{t}}{e_{r}}M_{n-r+t}[i.j]\,.
\]
If $d<0$, there exists $r>0$ such that $cr>|d|$. We have $M^{cn+d}=M^{c(n-r)+cr-|d|}$.
The claim follows similarly to the case of $d\geq0$ by multiplying
Eq. (\ref{eq:rec-ch}) by $M^{cr-|d|}$ and setting $t=n-r$. 
\end{proof}

\begin{lem}
\label{lem:c-finiteness-matrices-product} Let $r,m,\ell\in\mathbb{N}$
and let $\left\{ A_{n}\right\} _{n=n_{0}}^{\infty}$, $\left\{ B_{n}\right\} _{n=n_{0}}^{\infty}$
be C-finite matrix sequences of consisting of matrices of size $r\times m$
respectively $m\times\ell$ over $\mathbb{C}$. Then $A_{n}B_{n}$
is an C-finite matrix sequence. \end{lem}
\begin{proof}
Let $1\leq i\leq r$ and $1\leq j\leq\ell$. Then 
\[
\left(A_{n}B_{n}\right)_{ij}=\sum_{k=1}^{m}\left(A_{n}\right)_{ik}\left(B_{n}\right)_{kj}
\]
 is a polynomial in C-finite matrix sequences. Hence, by the closure
of C-finite sequences to finite addition and multiplication, $A_{n}B_{n}$
is an C-finite matrix sequences.
\end{proof}

\section{Proof of Theorem \ref{mainth:2} \label{se:mainth:2-proof}}

The proof of Theorem \ref{mainth:2} relies on the notion of a \emph{pseudo-inverse
of a matrix}. This notion is a generalization of the inverse of square
matrices to non-square matrices. For an introduction, see \cite{bk:BenIsraelGreville03}.
We need only the following theorem:
\begin{thm}[Moore-Penrose pseudo-inverse]
Let $\mathbb{F}$ be a subfield of $\mathbb{C}$. Let $s,t\in\mathbb{N}^{+}$.
Let $M$ be a matrix over $\mathbb{F}$ of size $s\times t$ with
$s\geq t$ whose columns are independent. Then there exists a unique matrix
$M^{+}$ over $\mathbb{F}$ of size $t\times s$ which satisfies the
following conditions:
\begin{enumerate}
\item $M^{*}M$ is non-singular;
\item $M^+=\left(M^{*}M\right)^{-1}M^{*}$; 
\item $M^{+}M=I$. 
\end{enumerate}

$M^{*}$ is the Hermitian transpose of $M$, i.e. $M^*$ is obtained by taking 
the transpose of $M$ and replacing each entry with its complex conjugate. 

$M^{+}$ is called the \emph{Moore-Penrose pseudo-inverse} of
$M$. 

\end{thm}

The following is the main lemma necessary for the proof of Theorem \ref{mainth:2}.
It allows to extract \Ct-recurrences for individual sequences of numbers 
from recursion schemes  with C-finite coefficients for multiple sequences of numbers. 

\begin{lem}
\label{th:matrix-to-rec-numbers} Let $\mathbb{F}$ be a subfield
of $\mathbb{C}$ and let $r\in\mathbb{N}^{+}$. For every $n\in\mathbb{N}^{+}$,
let $\overline{v_{n}}$ be a column vector of size $r\times1$ over
$\mathbb{F}$. Let $w_{n}$ be a C-finite sequence which is always
positive. Let $M_{n}$ be an C-finite \textup{matrix} sequence consisting
of matrices of size $r\times r$ over $\mathbb{F}$ such that, for
every $n$, 
\begin{gather}
\overline{v_{n+1}}=\frac{1}{w_{n}}M_{n}\overline{v_{n}}\,.\label{eq:vector-rec}
\end{gather}
For each $j=1,\dots,r$, $\overline{v_{n}}[j]$ is \Ctem-finite. Moreover,
all of the $\overline{v_{n}}[j]$ satisfy the same recurrence relation
(possibly with different initial conditions). \end{lem}
\begin{proof}
For every $i=0,\ldots,r^{2}$, let $M_{n}^{\{i\}}=\frac{1}{w_{n+i-1}\cdots w_{n-1}}M_{n+i-1}\cdots M_{n-1}$.
By Eq. (\ref{eq:vector-rec}), for every $n$, 
\begin{gather}
\overline{v_{n+i}}=M_{n}^{\{i\}}\overline{v_{n-1}}\,.\label{eq:vector-rec-1}
\end{gather}
Let $N_{n}^{\{0\}},\ldots,N_{n}^{\{r^{2}\}}$ be the column vectors
of size $r^{2}\times1$ corresponding to  $M_{n}^{\{0\}},
 $$\ldots,M_{n}^{\{r^{2}\}}$
with $N_{n}^{\{i\}}[r(k-1)+\ell]=M^{\{i\}}[k,\ell]$. For every fixed
$n$,  $N_{n}^{\{0\}},\ldots,N_{n}^{\{r^{2}\}}$ are members
of the vector space of column vectors over $\mathbb{F}$ of size $r^{2}\times1$.
Since this vector space is of dimension $r^{2}$, $N_{n}^{\{0\}},\ldots,N_{n}^{\{r^{2}\}}$
are linearly dependent. Let $s_{n}\in\{1,\ldots,r^{2}\}$ be such
that $N_{n}^{\{s_{n}\}},\ldots,N_{n}^{\{r^{2}\}}$ are linearly independent,
but  $N_{n}^{\{s_{n}-1\}},$$\ldots,N_{n}^{\{r^{2}\}}$
are linearly dependent. We have that $N_{n,s_{n}}^{*}N_{n,s_{n}}$ is non-singular.

For every $t=0,\ldots,r^{2}-1$ let $N_{n,t}$ be the $r^{2}\times(r^{2}-t)$
matrix whose columns are $N_{n}^{\{t\}},\ldots,N_{n}^{\{r^{2}-1\}}$.
Let
\[
\widetilde{N_{n,t}}= C\left(N_{n,t}^{*}N_{n,t}\right)^{T}N_{n,t}^{*}
\]
where $C\left(N_{n,t}^{*}N_{n,t}\right)^{T}$ is the transpose of
the cofactor matrix of $N_{n,t}^{*}N_{n,t}$.
Then 
\[N_{n,s_n}^{+}=\frac{1}{|N_{n,s_n}^{*}N_{n,s_n}|}\widetilde{N_{n,s_n}}=\left(N_{n,s_n}^{*}N_{n,s_n}\right)^{-1}N_{n,s_n}^{*}
\]
is the Moore-Penrose pseudo-inverse of $N_{n,s_{n}}$,
 where
$|N_{n,t}^{*}N_{n,t}|$ denotes the determinant of the matrix.
In particular, 
\begin{gather}
N_{n,s_{n}}^{+}N_{n,s_{n}}=I\,.\label{eq:Nplus_times_N_equals_I}
\end{gather}

Consider the system of linear equations 
\begin{gather}
N_{n,s_{n}}y_{n,s_{n}}=N_{n}^{\{r^{2}\}}\,.\label{eq:linear_sys}
\end{gather}
with $y_{n,s_n}$ a column vector of size $\left(r^{2}-s_{n}\right)\times1$
of indeterminates $y_{n,s_n}[k]$. 
Let \begin{eqnarray*}
y_{n,s_n}'&=&\widetilde{N_{n,s_n}}N_{n}^{\{r^{2}\}}\\
y_{n,s_{n}}&=&\frac{1}{\left|N_{n,s_{n}}^{*}N_{n,s_{n}}\right|}y_{n,s_{n}}'\,.
\end{eqnarray*}
Using Eq. (\ref{eq:Nplus_times_N_equals_I}) we have 
that $y_{n,s_n}$
is a solution of Eq. (\ref{eq:linear_sys}). This solution which can be rephrased
as the matrix equation:
\begin{gather}
y_{n,s_{n}}'[1]M_{n}^{\{s_{n}\}}+\cdots y_{n,s_{n}}'[r^{2}-s_{n}]M_{n}^{\{r^{2}-1\}}=\left|N_{n,s_{n}}^{*}N_{n,s_{n}}\right|M_{n}^{\{r^{2}\}}\,.\label{eq:matrix-equation}
\end{gather}
Moreover, by Lemmas \ref{lem:c-finiteness-matrices-product} and \ref{lem:c-finiteness-matrices},
$y_{n}'$ is an C-finite vector sequence. Multiplying Eq. (\ref{eq:matrix-equation})
from the right by $\overline{v_{n-1}}$ and rearranging, we get
\begin{gather}
y_{n,s_{n}}'[1]\overline{v_{n+s_{n}}}+\cdots y_{n,s_{n}}'[r^{2}-s_{n}]\overline{v_{n+r^{2}-1}}-\left|N_{n,s_{n}}^{*}N_{n,s_{n}}\right|\overline{v_{n+r^{2}}}=0\,,\label{eq:yprime-rec}
\end{gather}
For every $n$ and every $s\geq s_{n}$, $\left|N_{n,s}^{*}N_{n,s}\right|\not=0$,
and for every $s<s_{n}$, $\left|N_{n,s}^{*}N_{n,s}\right|=0$ are
linearly dependent. By Claim \ref{lem:c-finiteness-matrices}, there
exists $n_{1}$ such that for $n\geq n_{1}$, $s_{n}$ is periodic
and let $p$ be the period. Using this periodicity we can remove the
dependence of Eq. (\ref{eq:yprime-rec}) on the infinite sequence
$s_{n}$, and instead use for all $n\geq n_{1}$ a finite number of
values, $s_{n_{1}+1},\ldots,s_{n_{1}+p}$:
\begin{eqnarray*}
\sum_{i=1}^{p}\mathbb{I}_{n\equiv i\,(mod\, p)}\Bigg(\Bigg(\sum_{j=1}^{r^{2}-s_{n_{1}+i}}y_{n,s_{n_{1}+i}}'[j]\overline{v_{n+s_{n_{1}+i+j-1}}}\Bigg)\\
-\left|N_{n,s_{n_{1}+i}}^{*}N_{n,s_{n_{1}+i}}\right|\overline{v_{n+r^{2}}}\Bigg) & = & 0
\end{eqnarray*}
which can be rewritten as
\[
q_{n}^{\{0\}}\overline{v_{n}}+\cdots+q_{n}^{\{r^{2}-1\}}\overline{v_{n+r^{2}-1}}=q_{n}^{\{r^{2}\}}\overline{v_{n+r^{2}}}
\]
with 
\[
q_{n}^{\{t\}}=\begin{cases}
\sum_{i=1}^{p}\mathbb{I}_{n\equiv i\,(mod\, p)}y_{n,s_{n_{1}+i}}'[t+1-s_{n_{1}+i}],\,\,\, & 0\leq t\leq r^{2}-1\\
\sum_{i=1}^{p}\mathbb{I}_{n\equiv i\,(mod\, p)}\left|N_{n,s_{n_{1}+i}}^{*}N_{n,s_{n_{1}+i}}\right|,\,\,\, & t=r^{2}\,.
\end{cases}
\]
Note that, as the result of the closure of the C-finite sequences
to finite addition and multiplication, $q_{n}^{\{t\}}$ is C-finite.
Moreover, note $q_{n}^{\{r^{2}\}}$ is non-zero. 
\end{proof}

We can now turn the main proof of this section. 
\begin{proof}
[Proof of Theorem \ref{mainth:2}] Let $\zeta(n)=c\binom{n}{2}+dn+e$.
Let $b_{n}'=a_{\zeta(n)}$. We have 
\[
\zeta(n)-\zeta(n-1)=cn+d\,.
\]
 Let $a_{n}$ satisfy the C-recurrence
\begin{gather}
a_{n+s}=c^{(s-1)}a_{n+s-1}+\cdots+c^{(0)}a_{n}\,.\label{eq:rec-def-1}
\end{gather}
In order to write the latter equation in matrix form, let 
\[
M=\left(\begin{array}{ccc}
c^{(s-1)} & \cdots & c^{(0)}\\
1\\
 & \ddots\\
 &  & 1
\end{array}\right)
\]
where the empty entries are taken to be $0$. Let $\overline{u_{n}}=\left(a_{n},\ldots,a_{n-s+1}\right)^{tr}$.
We have 
\[
\overline{u_{n}}=M\overline{u_{n-1}}\,,
\]
 and consequently, 
\[
\overline{u_{\zeta(n)}}=M^{cn+d}\overline{u_{\zeta(n-1)}}\,.
\]
 For large enough values of $n$ such that $cn+d\geq0$, $M^{cn+d}$
is C-finite by Lemma \ref{lem:power-of-mat}. Hence, the desired result
follows from Lemma \ref{th:matrix-to-rec-numbers}. 
\end{proof}

As immediate consequences, we get closure properties for \Ct-finite
sequences over $\mathbb{C}$. 
\begin{cor}
Let $a_{n}$ and $b_{n}$ be \Ctem-finite sequences. The following
hold: 
\begin{enumerate}
\item $a_{n}+b_{n}$ is \Ctem-finite
\item $a_{n}b_{n}$ is \Ctem-finite
\end{enumerate}
\end{cor}
\begin{proof}
Let $a_{n}$ and $b_{n}$ satisfy the following recurrences

\begin{eqnarray*}
c_{n}^{(s)}a_{n+s} & = & c_{n}^{(s-1)}a_{n+s-1}+\cdots+c_{n}^{(0)}a_{n}\\
d_{n}^{(s)}b_{n+s'} & = & d_{n}^{(s'-1)}b_{n+s'-1}+\cdots+d_{n}^{(0)}b_{n}
\end{eqnarray*}
 where the sequences $c_{n}^{(i)}$ and $d_{n}^{(i)}$ are C-finite
and $c_{n}^{(s)}$ and $d_{n}^{(s)}$ are non-zero. It is convenient
to assume without loss of generality that $s=s'$. We apply Lemma
\ref{th:matrix-to-rec-numbers} for both cases:
\begin{enumerate}
\item For $a_{n}+b_{n}$, let $\overline{v_{n+1}}=\left(a_{n+1},\ldots,a_{n+2-s},b_{n+1},\ldots,b_{n+2-s}\right)^{tr}$
and 
\[
M=\frac{1}{c_{n}^{(s)}d_{n}^{(s)}}\left(\begin{array}{cccccc}
c_{n+1}^{(s-1)}d_{n}^{(s)} & \cdots & c_{n+2-s}^{(0)}d_{n}^{(s)} & d_{n+1}^{(s-1)}c_{n}^{(s)} & \cdots & d_{n+2-s}^{(0)}c_{n}^{(s)}\\
d_{n}^{(s)}\\
 & \ddots\\
 &  & d_{n}^{(s)}\\
 &  &  & c_{n}^{(s)}\\
 &  &  &  & \ddots\\
 &  &  &  &  & c_{n}^{(s)}
\end{array}\right)
\]
where the empty entries are taken to be $0$. We have $\overline{v_{n+1}}=M\overline{v_{n}}$
and the claim follows from Theorem \ref{th:matrix-to-rec-numbers}. 

\begin{enumerate}
\item For $a_{n}b_{n}$, we have
\begin{eqnarray}
c_{n}^{(s)}d_{n}^{(s)}a_{n+s}b_{n+s} & = & \sum_{t_{1},t_{2}=0}^{s-1}c_{n}^{(t_{1})}d_{n}^{(t_{2})}a_{n+t_{1}}b_{n+t_{2}}\,.\label{eq:multiplication-formula}
\end{eqnarray}
Let $\overline{v_{n+1}}=\left(a_{n+1-t_{1}}b_{n+1-t_{2}}:\,0\leq t_{1},t_{2}\leq s-1\right)^{tr}$.
Similarly to the case of $a_{n}+b_{n}$, we can define $M$ such that
$\overline{v_{n+1}}=\frac{1}{c_{n}^{(s)}d_{n}^{(s)}}M\overline{v_{n}}$,
where the first row of $M$ corresponds to Eq. (\ref{eq:multiplication-formula}),
and the subsequent rows consist of non-zero value and otherwise $0$s. 
\end{enumerate}
\end{enumerate}

\section{Fibonacci numbers \label{se:fibonacci}}
The Fibonacci number $F_{n}$, given by the famous recurrence 
\[
F_{n+2}=F_{n+1}+F_{n}
\]
with $F_{1}=1,\, F_{2}=1$, can also be described in terms of counting
binary words. $F_{n}$ counts the binary words of length $n-2$ which
do not contain consecutive $1$s. Similarly, $F_{n-1}$ counts the
binary words of length $n-2$ which begin with $0$ (or, equivalently,
end with $0$), and $F_{n-2}$ counts the binary words which begin
(end) with $1$. Let $W_{n}=F_{n+2}$.

Let $0<k<m$, then
\[
W_{m+k}=W_{k-1}W_{m}+W_{k-2}W_{m-1}
\]
since $W_{k-1}W_{m}$ counts the binary words of length $m+k$ with
no consecutive $1$s which have $0$ at index $k$, and $W_{k-2}W_{m-1}$
counts the binary words with no consecutive $1$s which have $1$
at index $k$(and therefore $0$at index$k-1$. This translates back
to the Fibonacci numbers as:
\[
F_{m+k+2}=F_{k+1}F_{m+2}+F_{k}F_{m+1}
\]
So we have for the appropriate choices of $m$ and $k$: 
\begin{eqnarray}
F_{(n+1)^{2}} & = & F_{2n+1}F_{n^{2}+1}+F_{2n}F_{n^{2}}\label{eq:Fib_nPlusOneSquare}\\
F_{n^{2}} & = & F_{2n}F_{(n-1)^{2}}+F_{2n-1}F_{(n-1)^{2}-1}\nonumber \\
F_{n^{2}+1} & = & F_{2n+1}F_{(n-1)^{2}}+F_{2n}F_{(n-1)^{2}-1}\nonumber 
\end{eqnarray}
Extracting $F_{(n-1)^{2}-1}$ from the second equation, we get: 
\begin{eqnarray*}
F_{(n-1)^{2}-1} & = & \frac{F_{n^{2}}-F_{2n}F_{(n-1)^{2}}}{F_{2n-1}}
\end{eqnarray*}
and substituting $F_{(n-1)^{2}-1}$ in the third equation, we have:
\begin{eqnarray*}
F_{n^{2}+1} & = & \frac{F_{2n}F_{n^{2}}+(F_{2n-1}F_{2n+1}-F_{2n}^{2})F_{(n-1)^{2}}}{F_{2n-1}}
\end{eqnarray*}
and substituting into Eq. (\ref{eq:Fib_nPlusOneSquare}), we have
\begin{eqnarray*}
F_{2n-1}F_{(n+1)^{2}} & = & F_{2n}\left(F_{2n+1}+F_{2n-1}\right)F_{n^{2}}+F_{2n+1}\left(F_{2n-1}F_{2n+1}-F_{2n}^{2}\right)F_{(n-1)^{2}}
\end{eqnarray*}
where $F_{2n-1}$, $F_{2n}\left(F_{2n+1}+F_{2n-1}\right)$ and $F_{2n+1}\left(F_{2n-1}F_{2n+1}-F_{2n}^{2}\right)$
are C-finite by the closure properties of C-finite sequences in Lemma
\ref{lem:closure}. Similarly, we can derive the following \Ct-recurrence
for $F_{\binom{n+1}{2}}$: 
\begin{eqnarray*}
F_{n-1}F_{\binom{n+1}{2}} & = & \left(F_{n-1}F_{n+1}+F_{n}F_{n-2}\right)F_{\binom{n}{2}}+\left(F_{n}F_{n-1}^{2}-F_{n-2}F_{n}^{2}\right)F_{\binom{n-1}{2}}
\end{eqnarray*}

The sequences $F_{n^2}$ and $F_{\binom{n}{2}}$ are catalogued in 
the On-Line Encyclopedia of Integer Sequences \cite{OEIS}
as (A054783) and (A081667).

\end{proof}

\section{Bi-iterative graph families \label{se:bi-iterative}}

In this section we define  the notion of a bi-iterative
graph family, give examples for some simple families which are bi-iterative 
and provide some simple lemmas for them. The graph families
we are interested in are built recursively by applying {\em basic operations}
on $k$-graphs. A {\em $k$-graph} is of the form
\[
G=\left(V,E;R_{1},\ldots,R_{k}\right)
\]
where $\left(V,E\right)$ is a simple graph and $R_{1},\ldots,R_{k}\subseteq V$
partition $V$. The sets $R_{1},\ldots,R_{k}$ are called {\em labels}.
The labels are used technically to aid in the description of the graph
families, but we are really only interested in the underlying graphs.
Before we give precise definitions and auxiliary lemmas for constructing
bi-iterative graph families, we give some examples of bi-iterative
graph families.

\begin{figure}
\begin{center}
\begin{tabular}{ccc}
\includegraphics[scale=0.8]{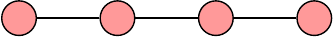}~~~~~~ & ~~~~~~\includegraphics[angle=90,scale=0.8]{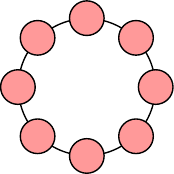}
~~~~~~ & ~~~~~~\includegraphics[angle=90,scale=0.8]{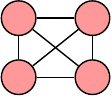}
\tabularnewline \, & &
\tabularnewline
\includegraphics[scale=0.8]{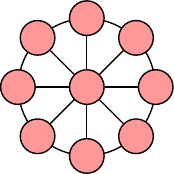}~~~~~~ & ~~~~~~\includegraphics[scale=0.8]{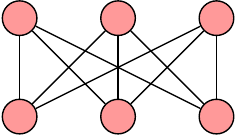}
~~~~~~ & ~~~~~~\includegraphics[angle=90,scale=0.8]{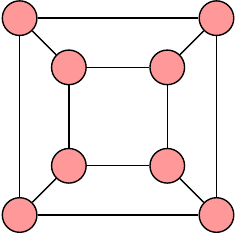}

\tabularnewline
\end{tabular}
\end{center}
\caption{\label{fig:iter-families} Examples of graphs belonging to the iterative
families: paths, cycles, cliques, wheels, complete bipartite graphs and prisms. They are also bi-iterative families. }
\end{figure}

\begin{figure}

\begin{tabular}{cc}
\includegraphics[scale=0.8]{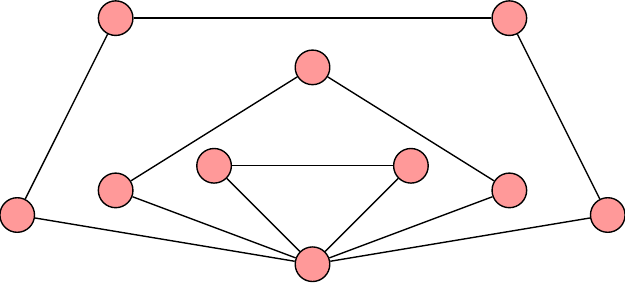}~~~~~~ &~~~~~~  \includegraphics[scale=0.8]{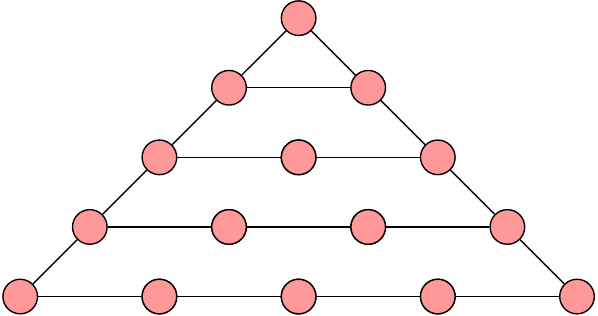}\tabularnewline
$G_{3}^{1}$ & $G_{3}^{4}$ \tabularnewline
 & \tabularnewline
\includegraphics[scale=0.8]{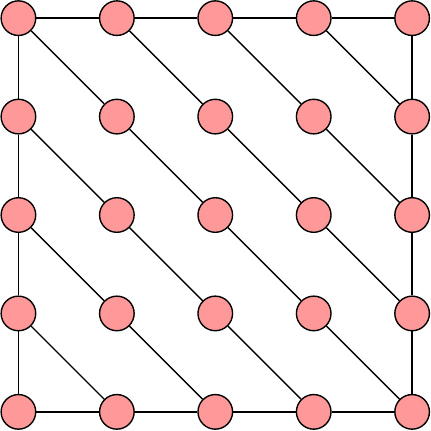} & \includegraphics{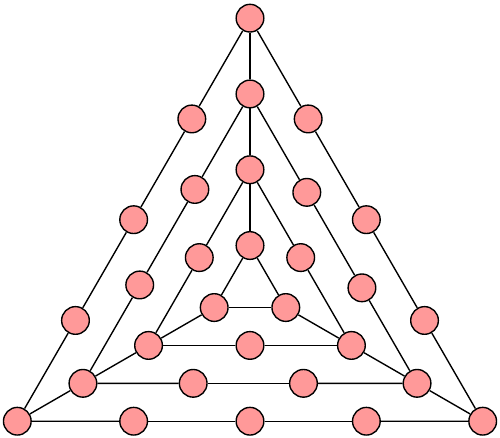}\tabularnewline
$G_{3}^{5}$ & $G_{3}^{6}$\tabularnewline
 & \tabularnewline
\end{tabular}

\begin{center}%
\begin{tabular}{c}
\includegraphics{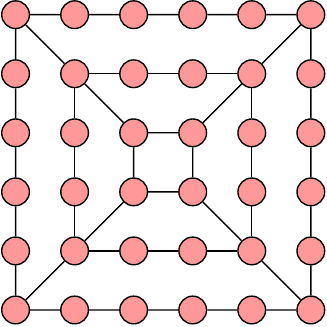}\tabularnewline
$G_{2}^{7}$\tabularnewline
\end{tabular}\end{center}

\caption{\label{fig:bi-families-bounded} Examples of graphs belonging to some of the bi-iterative
families of Example \ref{ex:families}. The bi-iterative families $G_n^1$, $G_n^4$, $G_n^5$ and $G_n^6$ are {\em bounded}.
}
\end{figure}

\begin{figure}
\begin{center}
\begin{tabular}{cc}
\includegraphics[angle=90,scale=0.8]{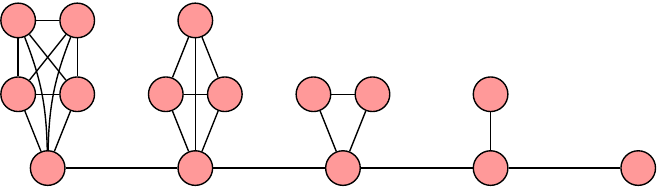}~~~ & ~~~\includegraphics[scale=0.8]{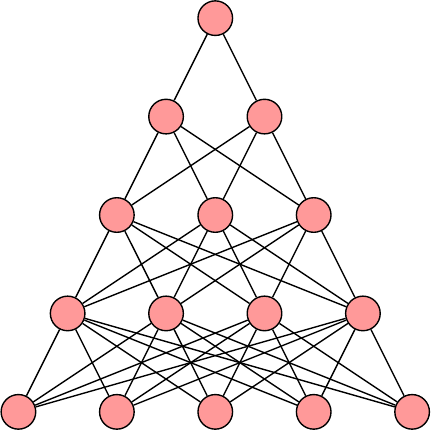}\tabularnewline
 $G_{4}^{2}$& $G_{4}^{3}$ \tabularnewline
\end{tabular}
\end{center}
\caption{\label{fig:bi-families-unbounded} Examples of graphs belonging to the bi-iterative
families of Example \ref{ex:families} not depicted in Figure \ref{fig:bi-families-bounded}. The bi-iterative families $G_n^2$ and $G_n^3$ 
are {\em not bounded}. }
\end{figure}

\begin{example}
[Bi-iterative graph families]\label{ex:families} See Figures \ref{fig:iter-families}, \ref{fig:bi-families-bounded}
and \ref{fig:bi-families-unbounded}
for illustrations of the following graph families. 
\begin{enumerate}
\item 
Iteratively families, such as paths, cycles, and cliques,
serve as simple examples of bi-iteratively constructible families.
\item $G_{0}^{1}$ is a single vertex labeled $2$. For each $n$, $G_{n}^{1}$
has one vertex labeled $2$ and all others are labeled $1$. $G_{n}^{1}$
is obtained from $G_{n-1}^{1}$ by adding a cycle of size $n+2$ and
identifying one vertex of the cycle with the vertex labeled $2$ is
in $G_{n-1}^{1}$. All other vertices in the cycle are labeled $1$. 
\item $G_{n}^{2}$ is obtained from a path of length $n+1$ by adding, for
each vertex $1\leq i\leq n+1$, a new clique of size $i$, and identifying
one vertex of the clique with $i$. 
\item $G_{0}^{3}$ is a single vertex labeled $2$. $G_{n}^{3}$ is obtained
from $G_{n-1}^{3}$ by adding $n+1$ isolated vertices labeled $3$,
adding all possible edges between vertices labeled $2$ and vertices
labeled $3$, relabel all from $2$ to $1$, and then from $3$ to
$2$.
\item $G_{0}^{4}$ consists of a triangle in which the vertices are labeled
$1$,$2$,$3$. $G_{n}^{4}$ is obtained from $G_{n-1}^{4}$ by adding
a path $P_{n+2}$ whose end-points are labeled $4$ and $5$. Then,
the edges $\{2,4\}$ and $\{3,5\}$ are added, and the labels are
changed so that the endpoints of the $P_{n+2}$ path are now labeled
$2$ and $3$, and all other vertices in $G_{n}^{4}$ are labeled
$1$. 
\item $G_{0}^{5}$ is obtained by taking two disjoint copies of $G_{0}^{4}$
and respectively identifying the vertices labeled $2$ and $3$. $G_{n}^{5}$
is obtained from two disjoint copies of $G_{n-1}^{4}$ by adding a
path $P_{n+2}$ and connecting each of its endpoints to the corresponding
end-points labeled $2$ and $3$ of the two copies of $G_{n}^{5}$. 
\item $G_{0}^{6}$ consists of a triangle in which the vertices are labeled
$2$,$3$,$4$. $G_{n}^{6}$ is obtained by adding to $G_{n-1}^{6}$
a cycle of size $3n+3$ in which three vertices are labeled $5$,$6$,$7$.
Between each of the pairs $(5,6)$, $(6,7)$ and $(5,7)$ there are
$n$ vertices labeled $1$. Then, $2,3,4$ are connected to $5,6,7$
respectively, and the labels are changed so that only the vertices
labeled $5,6,7$ remain labeled, and their new labels are $2,3,4$. 
\item The family $G_{n}^{7}$ is similar to $G_{n}^{6}$, except we add
a cycle of size $8n+4$, we have four distinguished vertices separated
by $n$ vertices labeled $1$, etc. 
\end{enumerate}
\end{example}
Now we proceed to define the precise definitions which allow us to
build such families. 
\begin{defn}
[Basic and elementary operations] \label{def:basic}~

The following are the {\em basic operations} on $k$-graphs:
\begin{enumerate}
\item $Add_{i}(G)$: A new vertex is added to $G$, where the new vertex
belongs to $R_{i}$;
\item $\rho_{i\to j}(G)$: All the vertices in $R_{i}$ are moved to $R_{j}$,
leaving $R_{j}$ empty; 
\item $\eta_{i,j}(G)$: All possible edges between vertices labeled $i$
and vertices labeled $j$ are added;
\item $\eta_{i,j}^{b}(G)$: If $R_{i}\cup R_{j}\leq b$, then $\eta_{i,j}^{b}(G)=\eta_{i,j}(G)$;
otherwise $\eta_{i,j}^{b}(G)=G$; 
\item $\delta_{i,j}(G)$: All edges between vertices labeled $i$ and vertices
labeled $j$ are removed. 
\end{enumerate}

An operation $F$ on $k$-graphs is {\em elementary} if $F$ is
a finite composition of any of the basic operations on $k$-graphs.
We denote by $id$ the elementary operation which leaves the $k$-graph
unchanged. 

\end{defn}

\begin{defn}
[Bi-iterative graph families] ~

Let $k\in\mathbb{N}$, $G_{0}$ be a $k$-graph and $F,H,L$ be elementary
operations on $k$-graphs, 
\begin{enumerate}
\item The sequence $F(G_{n}):\, n\in\mathbb{N}$ is called an $F$-iteration
family and is said to be an {\em iteratively constructible family}.
\item The sequence $G_{n+1}=H\left(F^{n}(L(G_{n}))\right):\, n\in\mathbb{N}$
is called an $(H,F,L)$-bi-iteration family and is said to be a {\em bi-iteratively constructible family}.
By $F^{n}(G)$ we mean the result of performing $n$ consecutive applications
of $F$ on $G$.
\end{enumerate}

Let $G_{n}:\, n\in\mathbb{N}$ be a family of graphs. This family
is (bi-)iteratively constructible if there exists $k\in\mathbb{N}$
and a family $G_{n}':\, n\in\mathbb{N}$ of $k$-graphs which is (bi-)iteratively
constructible, such that $G_{n}$ is obtained from $G_{n}'$ by ignoring
the labels. 

\end{defn}
It is sometimes convenient to describe $G_{0}$ using basic operations
on the empty graph $\emptyset$. 

We can now prove the observation from Example \ref{ex:families}(1):
\begin{lem}\label{lem:paths}
Every iteratively constructible family is bi-iteratively constructible.
\end{lem}
\begin{proof}
If $F$ is an elementary operation such that $G_{n}:\, n\in\mathbb{N}$
is an $F$-iteration family, then $G_{n}:\: n\in\mathbb{N}$ is also
an $(F,id,id)$-bi-iteration family.
\end{proof}

All of the families
in Example \ref{ex:families} are bi-iteratively constructible
families which are not iteratively constructible. The all grow 
too quickly to be iteratively constructible. 
Now consider for instance $G_{n}^{3}$. 
Let $F=Add_{3}$, $H=Add_{3}\circ\eta_{2,3}\circ\rho_{2\to1}\circ\rho_{3\to2}$
and $L=\emptyset$. We have $G_{n+1}^{3}=H(F^{n}(L(G_{n})))$.

In the sequel we will want to distinguish a particular type of bi-iterative
families, in which every application of $\eta_{i,j}$ only adds at
most a fixed amount of edges. 
\begin{defn}
[Bounded bi-iterative families] A basic operation is {\em bounded}
if it not of the type $\eta_{i,j}$. A bi-iteratively constructible
graph family $G_{n}:\, n\in\mathbb{N}$ is {\em bounded} if its construction
uses only bounded basic operations. \end{defn}
\begin{example}
Considering the families of Example \ref{ex:families}, it is not hard 
to see that $G_{n}^{1}$, $G_{n}^{4}$, $G_{n}^{5}$,$G_{n}^{6}$,$G_{n}^{7}$
are bounded bi-iterative families, while $G_{n}^{2}$,$G_{n}^{3}$
are bi-iterative families which are not bounded. 
\end{example}

\subsection{Lemmas for building bi-iterative graph families}
Here we give some lemmas which are useful to make the construction of bi-iterative families easier.
Their aim is to help the reader
understand which families of graph are bi-iterative. 
\begin{lem}
Let $G_{n}^{A},G_{n}^{B}:\, n\in\mathbb{N}$ be two bi-iteratively
constructible families. The family $G_{n}^{A}\sqcup G_{n}^{B}:\, n\in\mathbb{N}$
obtained by taking the disjoint union of the two families is bi-iteratively
constructible. In particular, if both families $G_{n}^{A},G_{n}^{B}:\, n\in\mathbb{N}$
are iteratively constructible, then so is $G_{n}^{A}\sqcup G_{n}^{B}:\, n\in\mathbb{N}$. \end{lem}
\begin{proof}
Let $H_{O},F_{O},L_{O}$ be elementary operations such that $G_{n}^{O}:\, n\in\mathbb{N}$
is an $(H_{O},F_{O},L_{O})$-bi-iteration family for $i=A,B$. We can
assume w.l.o.g. that the labels of the two families are disjoint;
if they are not, we can simply rename the labels used by one of the
families. The family $G_{n}^{A}\sqcup G_{n}^{B}:\, n\in\mathbb{N}$
is an $(H_{A}\circ H_{B},F_{A}\circ F_{B},L_{A}\circ L_{B})$-bi-iteration
family, where $\circ$ denotes the composition of operations. The
case in which $G_{n}^{A},G_{n}^{B}:\, n\in\mathbb{N}$ are iteratively
constructible is similar. 
\end{proof}

\begin{lem}
Let $G_{n}:\, n\in\mathbb{N}$ and $J_{n}:\, n\in\mathbb{N}$ be iteratively
constructible families of $k$-graphs whose basic operations use distinct
labels. The family $G_{n}\sqcup J_{n}$ is an iteratively constructible
family.\end{lem}
\begin{proof}
Let $F_{G}$ and $F_{J}$ be the elementary operations associated
with the two families. Let $F$ be the composition $F_{G}\circ F_{J}$.
The iteratively constructible family whose underlying elementary operation
is $F$ is $G_{n}\sqcup J_{n}$. 
\end{proof}

\begin{lem}
\label{lem:bi-from-i}Let $G_{n}:\, n\in\mathbb{N}$ be an iteratively
constructible family of $k$-graphs and let $H$ and $L$ be two elementary
operations over $k$-graphs. Let $D_{0}$ be a $k$-graph, and $D_{n+1}=H(L(D_{n})\sqcup G_{n})$.
The family $D_{n}:\, n\in\mathbb{N}$ is bi-iteratively constructible. \end{lem}
\begin{proof}
Let $F$ be an elementary operation such that $G_{n}:\, n\in\mathbb{N}$
is an $F$-iteration family. Let $F'$ and $G_{0}'$ be the same as
$F$ and $G_{0}$, except that the labels they use are changed as
follows. If a basic operation in $F$ uses label $i$, then the corresponding
operation in $F'$ uses label $i+k$. For every $i=1,\ldots,k$, let
$\rho_{i}=\rho_{i\to i+k}$. Let $\rho$ be the composition $\rho_{1}\circ\cdots\circ\rho_{k}$.
If a vertex in $G_{0}$ has label $i$, then the corresponding vertex
in $G_{0}'$ has label $i+k$. For every vertex $v$ of $G_{0}$ with
label $i$, let $a_{v}=Add_{i+k}$. Let $a$ be the composition of
$a_{v}$, $v\in V(G)$. We have $D_{n+1}=H(\rho(F'^{n}(a(L(D_{n})))))$,
and therefore $D_{n}:\, n\in\mathbb{N}$ is a bi-iteratively constructible
family of $2k$-graphs. 
\end{proof}

Using Lemma \ref{lem:bi-from-i}, it is easy to show that some families
from Example \ref{ex:families} are indeed bi-iterative.
\begin{example}
Consider $G_{n}^{4}$ from Example \ref{ex:families}. From Lemma \ref{lem:paths} we 
get that $\tilde{P}_{n}=P_{n+3}^{4,5}$ is an iterative
family. We define $E_{n+1}=H(L(E_{n})\sqcup\tilde{P}_{n})$ with $L=\rho_{1\to5}\circ\rho_{2\to6}$
and $H=\eta_{2,4}\circ\eta_{3,5}\circ\rho_{2\to1}\circ\rho_{3\to1}\circ\rho_{4\to2}\circ\rho_{3\to5}$.
We get $G_{n}^{4}=E_{n}$. 
\end{example}

Subfamilies of iteratively constructible families give rise to many
other related bi-iteratively constructible families:
\begin{lem}
Let $G_{n}:\: n\in\mathbb{N}$ be iteratively constructible. 
\begin{enumerate}
\item \textup{$G_{\binom{n}{2}}:\, n\in\mathbb{N}$ and $G_{n^{2}}:\, n\in\mathbb{N}$
are bi-iteratively constructible.}
\item Let $c\in\mathbb{N}^{+}$ and $d,e\in\mathbb{Z}$. There exists $r\in\mathbb{N}$
such that $H_{n}=G_{cm^{2}+dm+e}:\, m\in\mathbb{N},\, m=n+r$ is bi-iteratively
constructible. 
\end{enumerate}
\end{lem}
\begin{proof}
Let $F$ be an elementary operation such that $G_{n}:\, n\in\mathbb{N}$
is an $F$-iteration family. 
\begin{enumerate}
\item $G_{\binom{n}{2}}:\, n\in\mathbb{N}$ is an $(id,F,id)$-bi-iteration
family. The proof is by induction on $n$ with $G_{\binom{0}{2}}=G_{0}$
and 
\[
id\left(F^{n}\left(id\left(G_{\binom{n}{2}}\right)\right)\right)=F^{n+\binom{n}{2}}\left(G_{0}\right)=F^{\binom{n+1}{2}}\left(G_{0}\right)=G_{\binom{n+1}{2}}\,.
\]
$G_{n^{2}}:\, n\in\mathbb{N}$ is an $(id,F^{2},F)$-bi-iteration family.
Again by induction with $G_{0^{2}}=G_{0}$ and 
\[
id\left(F^{2n}\left(F(G_{n^{2}})\right)\right)=F^{2n+1+n^{2}}\left(G_{0}\right)=F^{(n+1)^{2}}\left(G_{0}\right)=G_{(n+1)^{2}}\,.
\]

\item Since $c>0$, there exists $r\in\mathbb{N}$ such that $c(n+r)^{2}+d(n+r)+e=cn^{2}+d'n+e'$
and $d',e'\geq0$. Let $H_{0}=G_{e'}$ , then $H_{n}:\, n\in\mathbb{N}$
is an $(id,F^{2c'},F^{d'+1})$-bi-iteration family. Here again the
proof is by induction on $n$. 
\end{enumerate}
\end{proof}

\subsection{Families which are not bi-iterative}
 {\em Clique-width} is a graph parameter
which generalizes tree-width, and is very useful for designing efficient
algorithms for NP-hard problems, see e.g. \cite{ar:CourcelleOlariu2000,ar:OumSeymour2006}. 

\begin{defn}
The clique-width $cwd(G)$ 
of a graph $G$ is the minimal $k\in\mathbb{N}$ such that
there exists  a $k$-graph $H$
whose underlying graph is isomorphic to $G$ and which can be 
obtained from $\emptyset$ by applying the basic operations $Add_i$, $\rho_{i\to j}$, $\eta_{i,j}$
and $\delta_{i,j}$ from Definition \ref{def:basic}. 
\end{defn}

Bi-iterative families have bounded clique-width. Using this fact we easily get examples of 
families which are not bi-iterative.

\begin{lem}
If $G_{n}:\: n\in\mathbb{N}$ is a bi-iterative family of $k$-graphs, then
for every $n$, $G_n$ has clique-width at most $k$.
\end{lem}
\begin{proof}
Let $G_n$ be a $(H,F,L)$-bi-iteration family of $k$-graphs. Since $G_0$ is a $k$-graph, 
it can be expressed by the basic operations $Add_i$, $\rho_{i\to j}$, and $\eta_{i,j}$ on $\emptyset$.
For every $n>0$, $G_n$ is a composition of the operations $H$, $F$ and $L$, which are in turn
compositions of basic operations. 
Therefore, for every $n$, $G_n$ can be obtained from $\emptyset$ by applying operations of the form
$Add_i$, $\rho_{i\to j}$, $\eta_{i,j}$, $\delta_{i,j}$, and $\eta_{i,j}^b$. 
It remains to notice that whenever an operation $\eta_{i,j}^b$ is applied to a $k$-graph $G'$, it can be either 
replaced by $\eta_{i,j}$ or omitted, depending on whether the number of vertices in $G'$ labeled $i$ or $j$ is smaller or equal to $b$ or not. 
Therefore, for every $n$, $G_n$ can be obtained from $\emptyset$ by applying operations of the form
$Add_i$, $\rho_{i\to j}$, $\eta_{i,j}$ and $\delta_{i,j}$ (but no operations of the form $\eta_{i,j}^b$). 
Therefore, each $G_n$ is of clique-width as most $k$. 
\end{proof}

Graph families which have unbounded clique-width, like square grids and other lattice graphs, 
are not bi-iterative. It is instructive to compare the graphs in Figure \ref{fig:non-bi-iterative} with the graphs of Figure \ref{fig:bi-families-bounded}. 

\begin{figure}
\begin{center}
\begin{tabular}{cc}
\includegraphics[angle=90,scale=0.8]{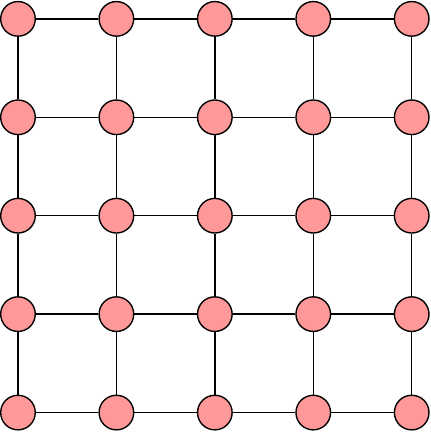}~~~ & ~~~\includegraphics[scale=0.8]{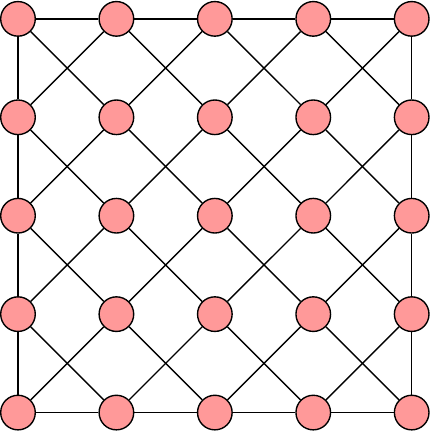}\tabularnewline
\end{tabular}
\end{center}
\caption{\label{fig:non-bi-iterative} Examples of graphs belonging to two families which are not bi-iterative, because they have unbounded clique-width.  }
\end{figure}

\section{Graph polynomials and \MSOL \label{se:msol}}

We consider in this paper two related rich families of graph polynomials with
useful decomposition properties. These graph polynomials are defined
using a simple logical language on graphs.

\subsection{Monadic Second Order Logic of graphs, \MSOL}

We define the logic \MSOLs of graphs inductively. We have three types
of variables: $x_{i}:\, i\in\mathbb{N}$ which range over vertices,
$U_{i}:\, i\in\mathbb{N}$ which range over sets of vertices and $B_{i}:\, i\in\mathbb{N}$
which range over sets of edges. We assume our graphs are {\em ordered},
i.e. that there exists an order relation $\leq$ on the vertices.
Atomic formulas are of the form $x_{i}=x_{j}$, $(x_{i},x_{j})\in E$,
$x_{i}\leq x_{j}$, $x_{i}\in U_{j}$ and $(x_{i},x_{j})\in B_{\ell}$.
The logical formulas of \MSOLs are built inductively from the atomic
formulas by using the connectives $\lor$ (or), $\land$ (and), $\neg$
(negation) and $\to$ (implication), and the quantifiers $\forall x_{i}$,
$\exists x_{i}$, $\forall U_{i}$, $\exists U_{i}$, $\forall B_{i}$,
$\exists B_{i}$ with their natural interpretation. 

If no variable $B_{i}$ occurs in the formula, then the formula is
said to be in \MSOL${}_G$, \MSOLs on graphs. Otherwise, the formula
is said to be on hypergraphs.%
\footnote{\MSOL$_{G}$ is referred to as node-\MSOLs in \cite{bk:Lovasz2012},
as \MS${}_1$ in \cite{bk:CourcelleEngelfriet2012}, and as \MSOL$(\tau_{graph})$
in \cite{pr:KotekMakowskyCSL12}. Full \MSOLs is sometimes referred
to as \MS${}_2$ or as \MSOL$(\tau_{hypergraph})$. $\tau_{graph}$
and $\tau_{hypergraph}$ are vocabularies whose structures represent
graphs in different ways, the later of which can also be used to represent
hypergraphs. %
} Sometimes additional modular quantifiers are allowed, giving rise
to the extended logic \CMSOL. The counting quantifiers are of the
form $C_{q}x\,\varphi(x)$, whose semantics is that the number of
elements from the universe satisfying $\varphi$ is zero modulo $q$.
On structures containing an order relation, as is the case here, \CMSOLs
and \MSOLs are equivalent, cf. \cite{bk:CourcelleEngelfriet2012}. 
\begin{example}
\label{ex:msol}~
\begin{enumerate}
\item We can express in \MSOLs that a set of edges $B_{1}$ is a matching:
\[
\varphi_{match}(B_{1})=\forall x_{1}\forall x_{2}\forall x_{3}\left((x_{1},x_{2})\in B_{1}\land(x_{2},x_{3})\in B_{1}\to x_{1}=x_{3}\right)
\]

\item We can express in \MSOLs that a set of vertices $U_{1}$ is an independent
set:
\[
\varphi_{ind}(U_{1})=\forall x_{1}\forall x_{2}\left((x_{1},x_{2})\in E\to\left(x_{1}\notin U_{1}\lor x_{2}\notin U_{1}\right)\right)
\]
where write e.g. $x_{1}\notin U_{1}$ as shorthand for $\neg\left(x_{1}\in U_{1}\right)$.
Note $\varphi_{ind}(U_{1})$ is a \MSOL${}_G$ formula.
\item A graphs is $3$-colorable iff it satisfies the following \MSOL${}_G$
formula:
\[
\exists U_{1}\exists U_{2}\exists U_{3}\left(\varphi_{partition}(U_{1},U_{2},U_{3})\land\varphi_{ind}(U_{1})\land\varphi_{ind}(U_{2})\land\varphi_{ind}(U_{3})\right)
\]
where $\varphi_{partition}$ expresses that $U_{1},U_{2},U_{3}$ form
a partition of the vertices: 
\begin{eqnarray*}
\varphi_{partition}(U_{1},U_{2},U_{3}) & = & \forall x_{1}\left(x_{1}\in U_{1}\lor x_{1}\in U_{2}\lor x_{1}\in U_{3}\right)\land\\
 &  & \forall x_{1}\neg\left(x_{1}\in U_{1}\land x_{1}\in U_{2}\right)\land\\
 &  & \forall x_{1}\neg\left(x_{1}\in U_{2}\land x_{1}\in U_{3}\right)\land\\
 &  & \forall x_{1}\neg\left(x_{1}\in U_{1}\land x_{1}\in U_{3}\right)
\end{eqnarray*}

\item We can express in \MSOLs that a vertex $x_{1}$ is the first element
is its connected component in the graph spanned by $B_{1}$ with respect
to the ordering of the vertices:
\[
\varphi_{fconn}(x_{1},B_{1})=\forall x_{2}\left(\varphi_{sc}(x_{1},x_{2})\to x_{1}\leq x_{2}\right)
\]
 where $\varphi_{sc}(x_{1},x_{2})$ says that $x_{1}$ and $x_{2}$
belong to the same connected component in the graph spanned by $B_{1}$:

\begin{center}$\begin{array}{rccl}
\varphi_{sc}(x_{1},x_{2},B_{1}) & = & \forall U_{1} & \Bigg(\bigg(x_{1}\in U_{1}\land x_{2}\notin U_{1}\bigg)\to\\
 &  &  & \exists x_{3}\exists x_{4}\left(B_{1}(x_{3},x_{4})\land x_{3}\in U_{1}\land x_{2}\notin U_{1}\right)\Bigg)
\end{array}$\end{center}The formula $\varphi_{fconn}(x_{1},B_{1})$ will be useful
when we discuss the definability of the Tutte polynomial. 

\end{enumerate}

\subsection{\MSOL-polynomials}

\MSOL-polynomials are a class of inductively defined graph polynomials
given e.g. in \cite{pr:GodlinKotekMakowsky08}. It is convenient to
refer to them in the following normal form:
\[
p=\sum_{U_{1},\ldots,U_{\ell},B_{1},\ldots,B_{m}:\Phi(\bar{U},\bar{B})}\,\, X_{1}^{\left|U_{1}\right|}\cdots X_{\ell'}^{\left|U_{\ell'}\right|}X_{\ell'+1}^{\left|B_{1}\right|}\cdots X_{\ell'+m'}^{\left|B_{m'}\right|}
\]
where $\Phi$ is an \MSOLs formula with the iteration variables indicated
and $\ell'\leq\ell$, $m'\leq m$. $\bar{U},\bar{B}$ is short for
$U_{1},\ldots,U_{\ell},B_{1},\ldots,B_{m}$. If $m=0$ and all the
formulas are \MSOL${}_G$ formulas, then we say $p$ is a \MSOL${}_G$-polynomial.
It is often convenient to think of the indeterminates $X_{i}$ as
\emph{multiplicative weights} of vertices and edges. 

While every \MSOL${}_G$-polynomial is a \MSOL-polynomial, the converse
is not true. The independence polynomial, the interlace polynomial
\cite{ar:Courcelle2008}, the domination polynomial and the vertex
cover polynomial are \MSOL$_G$-polynomials. The Tutte polynomial,
the matching polynomial, the characteristic polynomial and the edge
cover polynomial are \MSOL$_{HG}$. We illustrate this for the independence
polynomial and the Tutte polynomial. 
\end{example}

\subsection{The independence polynomial}

The independence polynomial is the generating function of independent
sets, 
\[
I(G)=\sum_{j=0}^{n}ind_{G}(j)X^{j}\,,
\]
where $ind_{G}(j)$ is the number of independent sets of size $j$
and $n$ is the number of vertices in $G$. It is a \MSOL${}_G$-polynomial,
given by 
\[
I(G)=\sum_{U_{1}:\Phi_{ind}(U_{1})}\,\, X^{|U_{1}|}
\]
where $\Phi_{ind}=\varphi_{ind}$ from Example \ref{ex:msol} says
$U_{1}$ is an independent set.

\subsection{The Tutte polynomial and the chromatic polynomial}

The chromatic polynomial is defined in terms of counting proper colorings,
but it can be written as a subset expansion which resembles an \MSOL-polynomial
as follows:
\begin{gather}
\chi(G)=\sum_{A\subseteq E}(-1)^{|A|}X^{k(A)}\label{eq:dichrom}
\end{gather}
where $k(A)$ is the number of connected components in the spanning
subgraph of $G$ with edge set $A$. 

Therefore, $\chi(G)$ is an evaluation of the dichromatic polynomial
given by 
\[
Z(G)=\sum_{A\subseteq E}Y^{|A|}X^{k(A)}
\]
 which is an \MSOL-polynomial:
\[
Z(G)=\sum_{U_{1},B_{1}:\Phi_{1}}\,\, Y^{|B_{1}|}X^{|U_{1}|}
\]
with $\Phi_{1}$ says that $U_{1}$ is the set of vertices which are
minimal in their connected component in the graph $(V,B_{1})$ with
respect to the ordering on the vertices
\[
\Phi_{1}=\forall x\left(x\in U_{1}\leftrightarrow\varphi_{fconn}(x_{1},B_{1})\right),
\]
where $\varphi_{fconn}$ is from Example \ref{ex:msol}. The dichromatic
polynomial is related to the Tutte polynomial via the following relation:
\[
T(G,X,Y)=\frac{Z(G,(X-1)(Y-1),Y-1)}{(X-1)^{k(E)}(Y-1)^{|V|}}\,.
\]
 The Tutte polynomial can also be shown to be an \MSOL-polynomial
via its definition in terms of spanning trees.

\subsection{A Feferman-Vaught-type theorem for \MSOL-polynomials}

The main technical tool from model theory that we use in this paper
is a decomposition property for \MSOL-polynomials, which resembles
decomposition theorems for formulas of First Order Logic, \FOL, and
\MSOL. For an extensive survey of the history and uses of Feferman-Vaught-type
theorems, including to \MSOL-polynomials, see \cite{ar:MakowskyTARSKI}.

In Theorem \ref{th:fv-poly} we rephrase Theorem 6.4 of
\cite{ar:MakowskyTARSKI}. For simplicity, we do not introduce the
general machinery that is used there, e.g. instead of the notion of
{\em\MSOL-smoothness of binary operations} we limit ourselves to
our elementary operations (see Section 4 of \cite{ar:MakowskyTARSKI}
for more details). Some other small differences follow from the proof
of Theorem 6.4.
\begin{thm}
[\cite{ar:MakowskyTARSKI}, see also \cite{pr:FischerMakowsky08}] \label{th:fv-poly}
Let $k$ be a natural number. Let $P$ be a finite set of \MSOL-polynomials.
Then there exists a finite set of \MSOL-polynomials $P'=\{p_{0},\ldots,p_{\alpha}\}$
such that $P\subseteq P'$ and for every elementary operation $\sigma$
on $k$-graphs, the following holds. If either all members of $P$
are \MSOL${}_G$-polynomials, or $\sigma$ consists only of bounded
basic operations, then there exists a matrix $M_{\sigma}$ such that
for every graph $G$, 
\[
\left(p_{0}(\sigma(G),\bar{X}),\ldots,p_{\alpha}(\sigma(G),\bar{X})\right)^{tr}=M_{\sigma}\left(p_{0}(G,\bar{X}),\ldots,p_{\alpha}(G,\bar{X})\right)^{tr}
\]
$M_{\sigma}$ is a matrix of size $\alpha\times\alpha$ of polynomials
with indeterminates $\bar{X}$. Additionally, if all members of $P$
are \MSOL$_G$-polynomials, then the same is true for $P'$. 
\end{thm}
For bi-iterative families of graphs we prove the following result,
which we will use in the proof of our main theorem. 
\begin{lem}
\label{lem:fv-bi-iterative} Let $k$ be a natural number. Let $p$
be an \MSOL-polynomial and let $G_{n}:\, n\in\mathbb{N}$ be a bi-iterative
graph family. If $p$ is an \MSOL$_G$-polynomial, or $G_{n}:\, n\in\mathbb{N}$
is bounded, then there exist a finite set of \MSOL-polynomials $P'=\{p_{0},\ldots,p_{\alpha}\}$
and a C-finite sequence $M_{n}:\, n\in\mathbb{N}$ such that $p\in P'$
and 

such that 
\[
\left(p_{0}(G_{n+1},\bar{X}),\ldots,p_{\alpha}(G_{n+1},\bar{X})\right)^{tr}=M_{n}\left(p_{0}(G_{n},\bar{X}),\ldots,p_{\alpha}(G_{n},\bar{X})\right)^{tr}
\]
Additionally, if $p$ is an \MSOL$_G$-polynomial, then the same is
true for all members of $P'$.
\end{lem}

\begin{proof}
Let $F$, $H$ and $L$ be elementary operations such that $G_{n+1}=H\left(F^{n}(L(G_{n}))\right)$.
Let $P'=\{p_{0},\ldots,p_{\alpha}\}$ be the set of \MSOL-polynomials
guaranteed in Theorem \ref{th:fv-poly} for $P=\{p\}$. We have 
\[
\left(p_{0}(\sigma(G),\bar{X}),\ldots,p_{\alpha}(\sigma(G),\bar{X})\right)^{tr}=M_{\sigma}\left(p_{0}(G,\bar{X}),\ldots,p_{\alpha}(G,\bar{X})\right)^{tr}
\]
for $\sigma\in\left\{ L,F,H\right\} $. Therefore, 
\[
\left(p_{0}(G_{n+1},\bar{X}),\ldots,p_{\alpha}(G_{n+1},\bar{X})\right)^{tr}=M_{H}M_{F}^{n}M_{L}\left(p_{0}(G,\bar{X}),\ldots,p_{\alpha}(G,\bar{X})\right)^{tr}\,.
\]
 By Lemmas \ref{lem:power-of-mat} and \ref{lem:c-finiteness-matrices-product},
$A_{n}=M_{H}M_{F}^{n}M_{L}$ is a C-finite sequence of matrices. 
\end{proof}

\section{Statement and proof of Theorem \ref{mainth:1} \label{se:mainth:1-proof}}

We are now ready to state Theorem \ref{mainth:1} exactly and prove
it. 
\begin{thm}
\label{th:main-formal} Let $k$ be a natural number. Let $p$ be
an \MSOL-polynomial and let $G_{n}:\, n\in\mathbb{N}$ be a bi-iterative
graph family. If $p$ is an \MSOL$_G$-polynomial, or $G_{n}:\, n\in\mathbb{N}$
is bounded, then the sequence $p(G_{n}):\, n\in\mathbb{N}$
is \Ctem-finite. 
\end{thm}

To transfer Theorem \ref{th:main-formal} to C-finite sequences over
a polynomial ring, we will use the following lemma:

\begin{lem}
\label{lem:injective} Let $\mathbb{F}$ be a countable subfield of
$\mathbb{C}$. For every $\xi\in\mathbb{N}$, there exists a set $D_{\xi}=\{d_{1},\ldots,d_{\xi}\}\subseteq\mathbb{R}$
such that the partial function $sub_{\xi}:\mathbb{F}[x_{1},\ldots,x_{\xi}]\to\mathbb{C}$
given by

\[
sub_{\xi}(p)=p(d_{1},\ldots,d_{\xi})
\]
is injective. \end{lem}
\begin{proof}
We prove the claim by induction on $\xi$. For the case $\xi=0$ we
have $D_{\xi}=\emptyset$ and $sub_{\xi}(p)=p$, which is injective. 

Now assume there exists $D_{\xi-1}$ such that $sub_{\xi-1}$ is injective.
Let $B_{\xi-1}$ be the set of real numbers which are roots of non-zero
polynomials in the polynomial ring $\mathbb{F}[d_{1},\ldots,d_{\xi-1}][x_{\xi}]$
of polynomials in the indeterminate $x_{\xi}$ whose coefficients
are polynomials in $d_{1},\ldots,d_{\xi-1}$ with rational coefficients.
The cardinality of $B_{\xi-1}$ is $\aleph_{0}$, implying that that
there exists $d_{\xi}\in\mathbb{R}\backslash B_{\xi-1}$. Let $D_{\xi}=D_{\xi-1}\cup\{d_{\xi}\}$.
Assume for contradiction that there exist distinct $p,q\in\mathbb{Q}[x_{1},\ldots,x_{\xi}]$
such that $sub_{\xi}(p)=sub_{\xi}(q)$. Let $r(x_{1},\ldots,x_{\xi})=p(x_{1},\ldots,x_{\xi})-q(x_{1},\ldots,x_{\xi})$.
Let 
\[
r(x_{1},\ldots,x_{\xi})=\sum_{i_{1},\ldots,i_{\xi}\leq t}\rho_{i_{1},\ldots,i_{\xi}}x_{1}^{i_{1}}\cdots x_{\xi}^{i_{\xi}}\,.
\]
Since $p$ and $q$ are distinct, $r$ is not the zero polynomial
and there exists $i_{\xi}'$ such that 
\[
r_{i_{\xi}'}(x_{1},\ldots,x_{\xi-1})=\sum_{i_{1},\ldots,i_{\xi-1}\leq t}\rho_{i_{1},\ldots,i_{\xi-1},i_{\xi}'}x_{1}^{i_{1}}\cdots x_{\xi-1}^{i_{\xi-1}}
\]
is not identically non-zero. 

By the assumption that $sub_{\xi}(p)=sub_{\xi}(q)$ we have that $r(d_{1},\ldots,d_{\xi})=0$. 
\begin{itemize}
\item If $x_{\xi}$ has non-zero degree in $r(d_{1},\ldots,d_{\xi-1},x_{\xi})$,
then $d_{\xi}$ is indeed a root of a non-zero polynomial $r(d_{1},\ldots,d_{\xi-1},x_{\xi})\in\mathbb{Q}[d_{1},\ldots,d_{\xi-1}][x_{\xi}]$. 
\item Otherwise, $r(d_{1},\ldots,d_{\xi-1},x_{\xi})$ is a polynomial of
degree zero in $x_{\xi}$. In order for $r(d_{1},\ldots,d_{\xi})=0$
to hold, $r(d_{1},\ldots,d_{\xi-1},x_{\xi})$ must be identically
zero. In particular, the coefficient of $x_{\xi}^{i_{\xi}'}$ in $r(d_{1},\ldots,d_{\xi-1},x_{\xi})$
is zero, but this coefficient is $r_{i_{\xi}'}(d_{1},\ldots,d_{\xi-1})$.
This implies that there exist two distinct polynomials, e.g. $r_{i_{\xi}'}(\bar{x})$
and $2r_{i_{\xi}'}(\bar{x})$, which agree on $d_{1},\ldots,d_{\xi-1}$
in contradiction to the assumption that $sub_{\xi-1}$ is injective. 
\end{itemize}
\end{proof}
\begin{lem}
\label{lem:polynomials} Let $\mathbb{F}$ be a subfield of $\mathbb{C}$
and let $r\in\mathbb{N}^{+}$ and let $r\in\mathbb{N}^{+}$. For every
$n\in\mathbb{N}^{+}$, let $\overline{v_{n}}$ be a column vector
of size $r\times1$ of polynomials in $\mathbb{F}[x_{1},\ldots,x_{k}]$.
Let $M_{n}$ be a C-finite sequence of matrices of size $r\times r$
over $\mathbb{F}[x_{1},\ldots,x_{k}]$ such that, for every $n$,
\begin{gather}
\overline{v_{n+1}}=M_{n}\overline{v_{n}}\,.\label{eq:vector-rec-3}
\end{gather}
For each $j=1,\dots,r$, $\overline{v_{n}}[j]$ is \Ctem-finite. Moreover,
all of the $\overline{v_{n}}[j]$ satisfy the same recurrence relation
(possibly with different initial conditions). 
\end{lem}
\begin{proof}
First note that due to the C-finiteness of $M_{n}$ and Eq. (\ref{eq:vector-rec-3}),
we may assume w.l.o.g. that the the matrices $M_{n}$ and vectors
$\overline{v_{n}}$ are all given over a finite extension field $\mathbb{F}$
of $\mathbb{C}$. In particular, we need that $\mathbb{F}$ is countable. 

Let $D_{k}=\{d_{1},\ldots,d_{k}\}$ be the set guaranteed in Lemma
\ref{lem:injective}. For every $n$, let $\overline{u_{n}}$ and
$L_{n}$ be the real vector respectively real matrix obtained from
$\overline{v_{n}}$ respectively $M_{n}$ by substituting $x_{1},\ldots,x_{k}$
with $d_{1},\ldots,d_{k}$. $L_{n}$ is a C-finite sequence of matrices
over $\mathbb{F}(d_{1},\ldots,d_{k})$ the extension field of $\mathbb{F}$
with $D_{k}$. We have for every $n$, 
\[
\overline{u_{n+1}}=L_{n}\overline{u_{n}}\,.
\]

By Lemma  \ref{th:matrix-to-rec-numbers}, there exists $n_{0}$ and
C-finite sequences over $\mathbb{F}(d_{1},\ldots,d_{k})$,\linebreak
$c_{n}^{\{0\}},\ldots,c_{n}^{\{r^{2}\}}$, such that for every $n>n_{0}$,
\[
c_{n}^{\{0\}}\overline{u_{n}}+\cdots+c_{n}^{\{r^{2}-1\}}\overline{u_{n+r^{2}-1}}=c_{n}^{\{r^{2}\}}\overline{u_{n+r^{2}}}
\]
and $q_{n}^{\{r^{2}\}}$ is non-zero. Using Lemma \ref{lem:injective},
there exist unique polynomials 
\[
q_{n}^{\{0\}}(x_{1},\ldots,x_{\xi}),\ldots,q_{n}^{\{r^{2}\}}(x_{1},\ldots,x_{\xi})
\]
 such that for every $n$, 
\[
q_{n}^{\{0\}}(d_{1},\ldots,d_{\xi})=c_{n}^{\{0\}}(d_{1},\ldots,d_{\xi})\,.
\]
Let $t(x_{1},\ldots,x_{\xi})$ be the polynomial given by 
\[
t(x_{1},\ldots,x_{\xi})=q_{n}^{\{0\}}\overline{v_{n}}+\cdots+q_{n}^{\{r^{2}-1\}}\overline{v_{n+r^{2}-1}}-q_{n}^{\{r^{2}\}}\overline{v_{n+r^{2}}}\,.
\]
substituting $d_{1},\ldots,d_{\xi}$ on both sides of the latter equation,
we get $sub_{\xi}(t)=0$, but this implies that $t(x_{1},\ldots,x_{\xi})$
is identically zero, since $sub_{\xi}(0)=0$ and $sub_{\xi}$ is injective. 

\end{proof}

\begin{proof}
[Proof of Theorem \ref{th:main-formal}] Let $P'=\{p_{0},\ldots,p_{\alpha}\}$
and $M_{n}:\, n\in\mathbb{N}$ be as guaranteed by Lemma \ref{lem:fv-bi-iterative}.
We have 
\[
\left(p_{0}(G_{n+1},\bar{X}),\ldots,p_{\alpha}(G_{n+1},\bar{X})\right)^{tr}=M_{n}\left(p_{0}(G_{n},\bar{X}),\ldots,p_{\alpha}(G_{n},\bar{X})\right)^{tr}\,.
\]
By Lemma \ref{lem:polynomials}, $p(G_{n}):\, n\in\mathbb{N}$ is
\Ct-finite. 
\end{proof}

\section{Examples of relatively iterative sequences \label{se:examples}}

Here we give explicit applications of Theorem \ref{mainth:1}. The
applications follow the basic ideas underlying the proof, but can
be significantly simplified given specific choices of a graph polynomial
and a bi-iterative family.

\subsection{The independence polynomial on $G_{n}^{2}$}

Let $G_{n}^{2}$ be as described in Example \ref{ex:families}. We
denote by $v_{0},\ldots,v_{n}$ the vertices of the underlying path
of $G_{n}^{2}$. Let $I_{A}(G_{n}^{2},x)$ ($I_{B}(G_{n}^{2},x)$)
be the generating functions counting independent sets $U_{1}$ in
$G_{n}^{2}$ such that $v_{n}$ belongs (resp. does not belong) to
$U_{1}$. Then, 
\begin{gather}
I(G_{n}^{2},x)=I_{A}(G_{n}^{2},x)+I_{B}(G_{n}^{2},x)\,.\label{eq:ind-sum}
\end{gather}
Now we give a matrix equation for computing $I_{A}(G_{n+1}^{2},x)$,$I_{B}(G_{n+1}^{2},x)$
and $I(G_{n+1}^{2},x)$ from $I_{A}(G_{n}^{2},x)$,$I_{B}(G_{n}^{2},x)$
and $I(G_{n}^{2},x)$: for all $m$, 
\begin{eqnarray}
\left(\begin{array}{c}
I_{A}(G_{m+1}^{2},x)\\
I_{B}(G_{m+1}^{2},x)
\end{array}\right) & = & M\left(\begin{array}{c}
I_{A}(G_{m}^{2},x)\\
I_{B}(G_{m}^{2},x)
\end{array}\right)\label{eq:matrix-ind}
\end{eqnarray}
 where 
\[
M=\left(\begin{array}{cc}
0 & x\\
1+nx & 1+nx
\end{array}\right)\,.
\]
The first row reflects the facts that if $v_{n+1}$ belongs to the
sets $U_{1}$ counted by $I_{A}(G_{n+1}^{2},x)$, $v_{n+1}$ and $v_{n}$
may not belong to the same $U_{1}$, and $v_{n+1}$ contributes a multiplicative factor of
$x$. The second row reflects that $v_{n+1}$ does not belong to
the sets $U_{1}$ counted in $I_{B}(G_{n+1}^{2},x)$, so independent
of whether $v_{n}$ is in $U_{1}$, there are two options: either
exactly one of the clique vertices adjacent to $v_{n+1}$ belong to
$U_{1}$ and contributes a factor of $x$, or no vertex of
that clique belongs to $U_{1}$, contributing a factor of $1$.

Eq. ({\ref{eq:matrix-ind}) holds both for $n$ and $n+1$, leading
to the recurrence relation 
\begin{eqnarray*}
I(G_{n+1}^{2},x) & = & (1+nx)I(G_{n}^{2},x)+x(1+(n-1)x)I(G_{n-1}^{2},x)\\
I(G_0^{2},x) & = & 1+x  \\ 
I(G_1^{2},x) & = & 1+3x+x^2   
\end{eqnarray*}
 using Eq. (\ref{eq:ind-sum}). This is a \Ct-finite recurrence, which
is also a P-recurrence. 

The number of independent sets of $G_{n+1}^2$ is $I(G_{n+1}^{2},1)$. Interestingly, 
the sequence $I(G_{n+1}^{2},1):n\in\mathbb{N}$ is in fact equal to the seemingly 
unrelated sequence (A052169) of \cite{OEIS}. This implies $I(G_{n+1}^{2},1)$
has an alternative combinatorial interpretation as the number of non-derangements of $1,\ldots,n+3$
divided by $n+2$. See \cite{ar:Petojevic2002} for a treatment of the related (A002467).

\subsection{\label{se:tutte-g4} The dichromatic polynomial on $G_{n}^{4}$}

Let $Z_{t}(P_{n+2})$ denote the dichromatic polynomial of $P_{n+2}$
such that the end-points of $P_{n+2}$ belong to the same connected
component iff $t=1$, for $t=0,1$. $Z_{t}(G_{n}^{4})$ is defined
similarly with respect to the most recently added path.

We have
\begin{eqnarray}
Z_{0}(G_{n}^{4}) & = & \left(\frac{v}{q}+1\right)^{2}Z_{0}(P_{n+2})\cdot Z_{0}(G_{n-1}^{4})\notag\\
 &  & +\left(2\frac{v}{q}+1\right)Z_{0}(P_{n+2})Z_{1}(G_{n-1}^{4})\notag\\
Z_{1}(G_{n}^{4}) & = & \frac{v^{2}}{q^{2}}Z_{0}(P_{n+2})Z_{1}(G_{n-1}^{4})\nonumber \\
 &  & +\left(\frac{v^{2}}{q}+2\frac{v}{q}+1\right)Z_{1}(P_{n+2})Z_{1}(G_{n-1}^{4})\nonumber \\
 &  & +\left(\frac{v}{q}+1\right)^{2}Z_{1}(P_{n+2})Z_{0}(G_{n-1}^{4})\notag
\end{eqnarray}
by dividing into cases by considering the end-points $u,v$ of $P_{n+2}$
and the end-points $u',v'$ of the $P_{n+1}$ in $G_{n-1}^{4}$ and
the edges $\{u,v\}$ and $\{u',v'\}$ with respect to the iteration
variable of $Z_{t}(G_{n}^{4})$. For example, the coefficient of $Z_{1}(P_{n})Z_{1}(G_{n-1}^{4})$
corresponds exactly to the case that $u,v$ are in the same connected
components in the graph spanned by $A$ ($A$ is the iteration variable
in the definition of $Z$ in Eq. (\ref{eq:dichrom})). If at least
one of the edges $\{u,v\}$ and $\{u',v'\}$ belongs to $A$, then
$G_{n-1}^{4}$ and $G_{n}^{4}$ have the same number of connected
components, but in $Z_{1}(P_{n+2})Z_{1}(G_{n-1}^{4})$ we have that
$Z_{1}(P_{n+2})$ contributes an additional factor of $q$ which should
be cancelled, so the weight in the case is $\frac{v^{2}+2v}{q}$.
If none of the two edges belongs to $A$, then $u,v$ are in a different
connected component from $u',v'$, so no correction is needed and
the weight is $1$. 

Using that $Z(G_{n}^{4})=Z_{0}(G_{n}^{4})+Z_{1}(G_{n}^{4})$, we get:

\begin{eqnarray}
Z_{0}(G_{n}^{4}) & = & \frac{v^{2}}{q^{2}}Z_{0}(P_{n+2})\cdot Z_{0}(G_{n-1}^{4})\notag\\
 &  & +\left(2\frac{v}{q}+1\right)Z_{0}(P_{n+2})Z(G_{n-1}^{4})\label{eq:z0gn}\\
Z(G_{n}^{4}) & = & \left(\left(\frac{v}{q}+1\right)^{2}Z(P_{n+2})+\frac{v^{2}}{q}\left(1-\frac{1}{q}\right)Z_{1}(P_{n+2})\right)Z(G_{n-1}^{4})\notag\\
 &  & -\frac{v^{2}}{q}\left(1-\frac{1}{q}\right)Z_{1}(P_{n+2})Z_{0}(G_{n-1}^{4})\label{eq:zgn}
\end{eqnarray}

Let $m\in\mathbb{N}$. Eqs. (\ref{eq:z0gn}) and (\ref{eq:zgn}) hold
for every $n$, in particular for $m$ and $m+1$, and from these
equations we can extract a recurrence relation for $Z(G_{m+1}^{4})$
using $Z(G_{m}^{4})$ and $Z(G_{m-1}^{4})$ by canceling out $Z_{0}(G_{m}^{4})$
and $Z_{0}(G_{m-1}^{4})$: 
\begin{eqnarray*}
Z(G_{m+1}^{4}) & = & Z(G_{m}^{4})\left(\left(\frac{v}{q}+1\right)^{2}Z(P_{m+3})+\frac{v^{2}}{q}\left(1-\frac{1}{q}\right)Z_{1}(P_{m+3})\right)\\
 &  & -Z(G_{n-1}^{4})\Bigg[\frac{v^{2}}{q}\left(1-\frac{1}{q}\right)Z_{1}(P_{m+3})\cdot\Bigg(Z_{0}(P_{m+2})\frac{v^{2}}{q^{2}}+\\
 &  & \frac{\left(\frac{v}{q}+1\right)^{2}Z_{0}(P_{m+2})Z(P_{m+2})}{\left(q-1\right)Z_{1}(P_{m+2})}+\left(2\frac{v}{q}+1\right)Z_{0}(P_{m+2})\Bigg)\Bigg]
\end{eqnarray*}
Using this recurrence relation, it is easy to compute the dichromatic
and Tutte polynomials. E.g., $Z(G_{m}^{4},3,-1)$, the number of $3$-proper
colorings of $G_{m}^{4}$, and $|Z(G_{m}^{4},-1,-1)|$, the number
of acyclic orientations of $G_{m}^{4}$, are given, for $m=0,\ldots,6$,
by

\begin{align*}
Z(G_{m}^{4},3,-1): & \mbox{6 30 318 6762 288354 24601830 4198550862}\\
|Z(G_{m}^{4},-1,-1)|: & \mbox{6 90 2826 179874 22988394 5882561010 3011536790874}
\end{align*}

\section{Conclusion and further research \label{se:conclusion}}

We introduced a natural type of recurrence relations, \Ct-recurrences,
and proved a general theorem stating that a wide class of graph polynomials
have recurrences of this type on some families of graphs. We gave
explicit applications to the Tutte polynomial and the independence
set polynomial. We further showed that quadratic sub-sequence of C-finite
sequences are \Ct -finite. 

A natural generalization of the notion of \Ct-recurrences could be
to allow even sparser sub-sequences. We say a sequence $a_{n}$ is
\emph{C$^{\,1}$-finite} if it is C-finite. We say a sequence is \emph{C$^{\,r}$-finite}
if it has a linear recurrence relation of the form 
\[
c_{n}^{(s)}a_{n+s}=c_{n}^{(s-1)}a_{n+s-1}+\cdots+c_{n}^{(0)}a_{n}
\]
where $c_{n}^{(0)},\ldots,c_{n}^{(s)}$ are C$^{r-1}$-finite. This
definition coincides with the definition of \Ct-finite.
\begin{problem}
\label{prob:a} Can we find families of graphs for which the Tutte polynomial and other
\MSOL-polynomials have C$^r$-recurrences?
\end{problem}

\bibliographystyle{plain}

\end{document}